\documentclass[twoside,english,reqno,a4paper, 10pt]{amsart}
\usepackage{listings,graphicx,amsmath,varioref,amscd,amssymb,color,bm,stmaryrd}

\usepackage{amsthm}

\newlength{\defbaselineskip}
\theoremstyle{plain}
\newtheorem{theorem}{Theorem}[section]
\newtheorem{proposition}[theorem]{Proposition}
\newtheorem{lemma}[theorem]{Lemma}

\theoremstyle{definition}
\newtheorem{definition}[theorem]{Definition}

\newtheorem{remark}[theorem]{Remark}

\theoremstyle{remark}

\newcommand{\re}{\mathbb{R}}

\def\rn{\mathbb{R}^{n}}

\def\be{\begin{equation}}
\def\ee{\end{equation}}
\def\rife#1{(\ref{#1})}

\def\ls{\mathcal{L}^{s}}

\def\t1p0{T^{1,p}_{0}(\Omega)}

\def\m2{M^{\frac{N(p-1)}{N-1}}(\Omega)}

\def\rn{\mathbb{R}^{n}}

\def\into{\int_{\Omega}}

\def\w-1p'{W^{-1,p'}(\Omega)}
\def\pw-1p'u{L^{p'}(0,1;W^{-1,p'}(\Omega))}

\def\lp'n{(L^{p'}(\Omega))^{N}}

\long\def\salta#1{\relax}

\begin{document}

\author[F. Petitta]{Francesco Petitta}
%
%
\address[F. Petitta]{Dipartimento di Scienze di Base e Applicate
per l' Ingegneria, ``Sapienza", Universit\`a di Roma, Via Scarpa 16, 00161 Roma, Italy.}\email{francesco.petitta@sbai.uniroma1.it}


\begin{abstract}

We deal with existence, uniqueness, and regularity for solutions of the  boundary value problem
$$
    \begin{cases}
       {\mathcal  L}^s u = \mu &\quad \text{in $\Omega$},  \\
        u(x)=0 \quad &\text{on} \ \ \rn\backslash\Omega,
    \end{cases}
$$
where $\Omega$ is a bounded domain of $\rn$,  $\mu$ is  a bounded radon measure on $\Omega$, and  ${\mathcal  L}^s$ is a nonlocal operator of fractional order $s$ whose kernel $K$ is comparable with the one of the factional laplacian.

\end{abstract}

\title[Nonlocal Equations with measure data]{Some remarks on the duality method for Integro-Differential  equations with measure data}


\maketitle


\section{Introduction}

In the last decade great attention has been paid to the theory of nonlocal operators, in particular integro-differential operators which naturally arise in a huge number of applications as for instance in  finance,  biology,  and physics.  The main feature of this type of operators is that they are suitable to be a model to 
anomalous diffusion process (as for instance Brownian Motions with jumps) or L\'evy Processes. For a gentle introduction to these issues we refer to \cite{v,jlv} and references therein.

Though the theory of nonlocal operators is nowadays highly developed (e.g. the regularity theory, see for instance \cite{bbck,bass,cs2} and related papers), our aim is to provide a general unified framework to existence, uniqueness and weak regularity for general  boundary values problems involving nonlocal operators on bounded domains and highly irregular data (namely measures).  

To be more concrete let $\Omega$ be a bounded domain of $\rn$ and let us consider the following boundary value problem
\begin{equation}\label{maini}
    \begin{cases}
       {\mathcal  L}^s u = \mu &\quad \text{in $\Omega$},  \\
        u(x)=0 \quad &\text{on} \ \ \rn\backslash\Omega,
    \end{cases}
\end{equation}
where $\mu$  a bounded Radon measure on $\Omega$ and  ${\mathcal  L}^s$ is a nonlocal operator of fractional order $s$, for instance (we will be more precise here below), let ${\mathcal  L}^s$ be  an integro-differential operator  with kernel $K$ such that $$ 
K(x,y)\sim \frac{1}{|x-y|^{n+2s}}\,.
$$

Following the idea introduced in \cite{kpu} (and inspired by \cite{s}) for fractional laplacian type problems in $\re^{n}$ we will introduce the notion of duality solutions for problems  as \rife{maini} and we will prove the existence of a unique solution for this problem with minimal assumptions on both the domain and the data.

We want to stress that some of the results we will present here can be found (or deduced) spread in the current literature. In particular,  the case ${\mathcal  L}^s=(-\Delta)^{s}$ (i.e. is the fractional laplacian of order $s$) and $\Omega=\rn$, as we already mentioned, was treated in \cite{kpu}, while Dirichlet boundary values problems with smooth data in bounded domains was considered for instance in \cite{chso, rs2, rs}. Again in the fractional laplacian case, boundary value problems with measure data in smooth domains (namely $C^{2}$) was recently considered in \cite{cv}.

For more general operators the theory is less complete and, besides the classical books by E. Stein (\cite{st}) and N. S. Landkof (\cite{lan}),  we refer the reader to \cite{ka,dk,fkv,lpps,bpv}. We also mention the recent \cite{kms} in which the regularity theory for the so called SOLA solutions .  

As we said, one of our goals is to provide a common framework for all these previous results: we will indeed consider general nonlocal operators on bounded domains and measure data. We will address  the question  of existence, uniqueness and weak (e.g. fractional) regularity for solution to Dirichlet problems as in \rife{maini}.

The paper is organized as follows:  in Section \ref{2} we give our concept of solution for rather general  integro-differential boundary value problems, we investigate the relation of such a concept with the classical weak notion, and we state and prove the existence and uniqueness of a duality solutions. Section \ref{frac} will be devoted to the discussion of  fractional sobolev regularity for duality solutions as well as further remarks and comments.

\section{General Integro-Differential equations}\label{2}

In this section we describe and prove our main existence and uniqueness result. As we shall stress later the Stampacchia's duality method we are going to use   is robust enough to be applied to very general linear nonlocal operators (see also Remark \ref{meta} below). Despite of this fact,   for the sake of concreteness we will develop our argument in the (anyhow large) class of symmetric integro-differential operators. In particular,  for $s\in(0,1)$, we are interested in solving the following Dirichlet boundary value problem:

\begin{equation}\label{mainide}
\left\{ \begin{array}{rcll}
\mathcal{L}^{s} u &=&\mu&\textrm{in }\Omega \\
u&=&0&\textrm{in }\rn\backslash\Omega,
\end{array}\right.
\end{equation}
where $\mu$ is any bounded radon measure and $\Omega$ is any bounded domain of $\rn$, $n\geq2$.
The operator $\ls$ is given by
\begin{equation}
\ls u(x)= {\rm PV}\int_{\rn}\bigl(u(x)-u(x+y)\bigr)K(y)dy,
\end{equation}
where $K$ is a nonnegative  kernel satisfying
\[\frac{ \lambda}{|y|^{n+2s}}\leq K(y)\leq \frac{\Lambda}{|y|^{n+2s}},\]
with $0<\lambda\leq \Lambda$.

This class of kernels is, for instance, the one considered by Caffarelli-Silvestre in \cite{cs2, cs3} (see also \cite{dk,lpps}), and it contains as a particular case the fractional laplacian case, that is 
\begin{equation}\label{fracche}
(-\Delta )^{s} u(x):= c_{n,s}{\rm PV}\int_{\rn}\frac{u(x)-u(x+y)}{|y|^{n+2s}}dy,
\end{equation}
with
$$
c_{n,s}=\frac{4^{s}s(s-1)\Gamma(\frac{n+2s}{2})}{\pi^{\frac{n}{2}}\Gamma(2-s)}\,,
$$
where $\Gamma$ is the Euler Gamma Function.

Here is our notion of solution which extends the one given in \cite{s} (see also \cite{kpu}).
\begin{definition}\label{dualei}
We say that  a function $u\in L^1(\rn)$, is
a \emph{duality solution} for problem \rife{mainide} if $u\equiv0$ in $\rn\setminus\Omega$ and
\begin{equation}\label{dualfi}
\int_{\Omega} u g \,dx= \int_{\Omega} w d\mu,
\end{equation}
for all $g\in C^{\infty}_{0}(\Omega)$, where $w$ is the weak solution of
\begin{equation}\label{duali}
\left\{ \begin{array}{rcll}
\ls w &=&g&\textrm{in }\Omega \\
w&=&0&\textrm{in }\rn\backslash\Omega\,.
\end{array}\right.
\end{equation}
 \end{definition}

\begin{remark}
Some remarks are in order to be done here. First of all the existence of a weak solution (whose definition will be recalled later, see Definition \ref{weak} below) for problem \rife{duali} is an easy consequence of Lax-Milgram lemma (see for instance \cite{fkv} or \cite{lpps}). 

 In order for all terms in \rife{dualfi} to be  well defined we need  $w\in C(\overline{\Omega})$ (or, at least,  $w\in L^{\infty}(\Omega)$ if $\mu \in L^{1}(\Omega)$); this will be showed in Lemma \ref{cont} below. 

 We refer to Section \ref{remare} below for further comments on this definition, the relations with the other available definitions and to the sense in which the boundary data is assumed.  
\end{remark}

\medskip

Before stating our main result we need to recall the following definition

\begin{definition}
We say that  $\Omega$ satisfies  the uniform \emph{exterior ball condition} if there exists $s_\Omega>0$ such that for every $x\in\rn\backslash\Omega$ with $dist(x,\partial\Omega)<s_\Omega$, there is $z_x\in\partial\Omega$ such that $|x-z_x|=dist(x,\partial \Omega)$ and $B(x_0,s_\Omega)\subset\rn\setminus\Omega$ with $x_0:=z_x+s_\Omega\frac{x-z_x}{|x-z_x|}$.  In an analogous way, one can define the  uniform \emph{interior ball condition} by replacing $\rn\setminus\Omega$ with $\Omega$.  
\end{definition}

In order to better understand the previous definition, let us recall that, as is proved in  \cite[Corollary 3.14]{AlMaz}, a domain with compact boundary is of class $\mathcal{C}^{1,1}$ if and only if it satisfies both a uniform interior ball condition and an exterior one. 
\medskip

Here is our existence and uniqueness result

\begin{theorem}\label{tei}
Let $\Omega$ be a bounded open set satisfying the uniform exterior ball condition. 
Then there exists a unique duality solution  for problem \rife{mainide} in the sense of Definition \ref{dualei}. Moreover, $u\in L^q(\rn)$, for any $q<\frac{n}{n-2s}$.
\end{theorem}

\begin{remark}
As we will deduce from the proof of Theorem \ref{tei} the regularity condition on $\Omega$ is essential in order to get continuity estimates up to the boundary for the solution. In the case of a general bounded domain, if $\mu\in L^{1}(\Omega)$, with the same proof, one can obtain   existence and uniqueness  of a duality solution for problem \rife{mainide} which turns out to coincide with the one obtained in \cite{lpps}. 

Also observe that, as already noticed in \cite{kpu} for the solutions in $\rn$, the regularity of the solution we find is optimal as the fundamental solution (i.e. the solution with datum $\mu=\delta_0$) for these operators is comparable to
$|x|^{2s-n}$ near the origin.  Finally observe that, as $s$ goes to $1^{-}$, we formally obtain the sharp regularity of \cite{s} for boundary value problems involving linear second order differential equations. This latter fact is completely formal as, as $s$ approaches $1$ nonlocal integro-differential operators may degenerate. 

\end{remark}

\subsection{Some useful notations and tools}

We will made use of some basic results concerning fractional Sobolev spaces (also called spaces of Bessel potentials). For a review in the subject we refer to \cite{ada,st} (see also \cite{dpv}). Let us  recall the following 

\begin{definition}
For $0<s< 1$ and $1\leq p<\infty$. We define the \emph{fractional sobolev space}
$W^{s,p}(\rn)$ as the set of all functions $u$ in $L^p (\rn)$
such that
 $$
 \int_{\rn}\int_{\rn} \frac{|u(x)-u(y)|^p}{|x-y|^{n+s p}}\ dxdy<\infty\,,
 $$
endowed with the norm
\begin{equation}\label{norm}
\|u\|_{W^{s,p}(\rn)}=\|u\|_{L^p (\rn)}
+ \left( \int_{\rn}\int_{\rn} \frac{|u(x)-u(y)|^p}{|x-y|^{n+s p}}\ dxdy\right)^{\frac{1}{p}}\,.
\end{equation}
If $p=2$ we will use the usual notation $W^{s,p}(\rn)=H^{s}(\rn)$. 
\end{definition}

If $\Omega$ is any bounded domain of $\rn$ the space $W^{s,p}(\Omega)$ is defined in a similar way, while  $
W_{0}^{s,p}(\Omega)$ is defined as the closure of $C^{\infty}_{0}(\Omega)$ with respect to the norm defined in \rife{norm}. 

The following embedding theorem is valid in domains satisfying the so-called \emph{extension property} which, roughy speaking, consists in the fact that functions in $W^{s,p}(\Omega)$ can be extended to functions in $W^{s,p}(\rn)$. Concretely one can think, for instance, of $\Omega$ to be an open domain with lipschitz boundary.

We recall the fractional Sobolev embedding result we shall use (see \cite{ada}).
\begin{theorem}[Sobolev embedding]\label{emb}
Let $\Omega$ be a bounded domain of $\rn$ with the extension property and let $sp<n$.
Then, there exists a constant $C$ depending only on $s$ and $n$, such that
$$
    \|v\|_{L^{{p^{*}_{s}}}(\Omega)}\leq C\|v\|_{W^{s, p}(\Omega)},
    \qquad \forall v\in W^{s, p}(\Omega),
$$
 where $p^{*}_{s}=\frac{np}{n-s p}$ is the fractional Sobolev critical exponent, $p>1$ and $0<s<1$. 
 
 Moreover, if  $\Omega$ has no external cusps and $ps>n$, then 
 $$
 \|v\|_{C^{\gamma}(\Omega)}\leq C\|v\|_{W^{s,p}(\Omega)}\,,     \qquad \forall v\in W^{s, p}(\Omega), 
 $$
where $\gamma=\frac{sp-n}{p}$.
\end{theorem}

\subsection{Weak solution vs Duality Solution}
\label{remare}

Due to the generality of the results we presented, some important features of the duality formulation we introduced could be be missed. At first glance, in such a nonregular framework, one could ask weather the duality formulation is the good one in order to give sense to the boundary datum $u=0$ on $\rn\backslash\Omega$. 

In order to give some insights towards  the question of the boundary datum we restrict ourself to the toy model of the fractional laplacian case. The reader will easily deduce that the following arguments keep working in more general cases as an integration by parts formula in all of $\rn$ is available for functions in $H^{s}(\rn)$.   

Let $\Omega$ be a bounded open set of $\rn$ and let us consider the following problem 
\begin{equation}\label{main}
    \begin{cases}
        (-\Delta)^s u = \mu &\quad \text{in $\Omega$},  \\
        u(x)=0 \quad &\text{on} \ \ \rn\backslash\Omega\,,
    \end{cases}
\end{equation}
where $\mu$ is a bounded  Radon measure on  $\Omega$ of $\rn$, and $ (-\Delta)^s$ is the fractional laplace operator introduced in \rife{fracche} with  $s\in (0,1)$.

As we already noticed, for data regular enough  (namely if  $\mu\in H^{-s}(\Omega)$), then existence of a finite energy solution (i.e. weak solutions in $H^{s}(\Omega)$) for problem  \rife{main} is an easy consequence of Lax-Milgram Lemma (see for instance \cite{lpps}). Existence, uniqueness and regularity up to the boundary for solutions to problem \rife{main} can be found  in \cite{rs2} in the case of bounded data $\mu\in L^{\infty}(\Omega)$. 

In Definition \ref{dualei} we added the boundary condition $u=0$ on $\rn\backslash\Omega$. This  seems to be quite artificial  and it need to be better explained. If $u$ satisfies \rife{dualfi} then, at least formally, using the integration by parts formula in all $\rn$, we have  
    \[\begin{split} \int_\Omega w\,d\mu=\int_\Omega u(-\Delta)^sw&=\int_{\rn} u(-\Delta)^sw-\int_{\rn \setminus\Omega}u(-\Delta)^sw\\
    &=\int_{\rn} (-\Delta)^s u\,\,w-\int_{\rn \setminus\Omega}u(-\Delta)^sw\\
    &=\int_{\Omega} w\,(-\Delta)^su-\int_{\rn \setminus\Omega}u(-\Delta)^sw,\end{split}\]
    where in the last equality we have used $w\equiv0$ in $\rn\setminus\Omega$.

The previous identity can be recast as
    \begin{equation}\label{u}
    \int_\Omega w\,d\mu=\int_{\Omega} w\,(-\Delta)^su+\int_\Omega\int_{\rn \setminus\Omega}\frac{w(x)u(y)}{|x-y|^{n+2s}}\,dx\,dy.
    \end{equation}
In particular \rife{u} is satisfied  for all $w\in C^\infty_c(\Omega)$. That is,
    \[(-\Delta)^su(x)+\int_{\rn \setminus\Omega}\frac{u(y)}{|x-y|^{n+2s}}\,dy=\mu(x)\,, \ \text{in}\ \mathcal{D}'(\Omega). \]
    The second term is a function of $x$ that depends on the values of $u$ outside. This  suggest to impose  $u=0$ on $\rn\backslash\Omega$.  Proposition \ref{equiv} below clarifies that this is the right choice. 
First of all we need the following definition (see \cite{rs2}) recalling that, by Theorem \ref{emb}, $L^{(2^{*}_{s})'} (\Omega)\subset H^{-s}(\Omega)$. 

\begin{definition}\label{weak}
Let $\mu\in L^{(2^{*}_{s})'} (\Omega)$. A weak solution for problem \rife{main} is a function $u\in H^{s}(\rn)$ such that $u\equiv 0 $ a.e. on $\rn\backslash\Omega$, and
$$
\int_{\rn}(-\Delta)^{\frac{s}{2}}u(-\Delta)^{\frac{s}{2}}v\ dx= \int_{\Omega}\mu v\ dx
$$
for any $v\in H^{s}(\rn)$ such that $v\equiv 0 $ (a.e.) on $\rn\backslash\Omega$.
\end{definition}
\begin{proposition}\label{equiv}
Let $\mu\in L^{(2^{*}_{s})'} (\Omega)$ then $u$ is a duality solution of problem \rife{main} in the sense of Definition \ref{dualei} if and only if $u$ is a weak solution of problem \rife{main}.
\end{proposition}
\begin{proof}
Let $u$ be a  weak solution of problem \rife{main}, and let $g\in C^{\infty}_{0}(\Omega)$.   If $w$ be the corresponding solution with datum $g$.
    Then, integrating by parts and using that $u=0$ outside $\Omega$, we have
$$
\begin{array}{l}
\displaystyle\int_{\Omega} u g\,dx=
\int_{\Omega}u(-\Delta)^s w\,dx=\int_{\rn}u(-\Delta)^s w\,dx\\\\  \displaystyle-\int_{\rn\backslash\Omega}u(-\Delta)^s w\,dx
= \int_{\rn}(-\Delta)^{\frac{s}{2}} u (-\Delta )^{\frac{s}{2}} w dx=\int_{\Omega}\mu w\,dx\,.
\end{array}
$$
Now, let $u$ be the duality solution of problem \rife{main} and let $\tilde{u}$ be the weak solution of the same problem. Then, reasoning as before, we have
$$
\int_{\Omega} \tilde{u} g\,dx=\int_{\Omega}\mu w\,dx
$$
for any $g\in C^{\infty}_{0}(\Omega)$. Then, subtracting the formulation of $u$ we get
$$
\int_{\Omega}(u- \tilde{u}) g\,dx= 0,
$$
for any $g\in C^{\infty}_{0}(\Omega)$, that implies $u=\tilde{u}$.
\end{proof}

\begin{remark}\label{valdi}
The previous result shows the equivalence between the duality and the weak formulations in the case of finite energy solutions for homogeneous problems in bounded domains.  In the case of general nonhomogeneous  integro-differential problems of the form
\begin{equation}\label{phi}
\left\{ \begin{array}{rcll}
\ls u &=&0&\textrm{in }\Omega \\
u&=&\phi&\textrm{in }\rn\backslash\Omega,
\end{array}\right.
\end{equation}
it is worth introducing the further notion of viscosity solution at least for smooth data $\phi$ (see for instance \cite{cs2, barles}).

For instance, let $\phi$ be a bounded function in $C^{\gamma}(\rn\backslash\overline{\Omega})$ for some positive $\gamma<1$, and let $\Omega$ satisfying the uniform exterior ball condition.  The existence of a unique viscosity solution for problem \rife{phi} can be proved by Perron's method in a standard way through the construction of suitable barriers (see \cite{cs3} and \cite{kl}). As pointed out in \cite{rs2} (see also \cite{sv}) the unique viscosity solution for problem \rife{phi} turns out to coincides with the weak one due to both the  interior regularity for viscosity solutions (see for instance \cite{cs2}) and the existence and uniqueness of weak solution for the same problem (see for instance \cite{fkv}).
\end{remark}

\subsection{Existence and Uniqueness of a duality solution}
Before proving Theorem \ref{tei} we need a preliminary result which is based on the main result in \cite{dk}. In this paper, the authors proved that if
\[\ls v=h\quad \text{in}\ \rn,\]
and $h\in L^p(\rn)$, with $p>1$, then
$$
\|v\|_{W^{2s,p}(\rn)}\leq C\|h\|_{L^p(\rn)}\,.
$$
By the Sobolev embedding, this implies that
\begin{equation}\label{dk}\|v\|_{C^\gamma(\rn)}\leq C\|h\|_{L^p(\rn)}\quad whenever\ p>\frac{n}{2s}.\end{equation}

\begin{lemma} \label{cont}
Let $\Omega$ be any bounded domain, $p>\frac{n}{2s}$, $f\in L^p(\Omega)$, and $u$ be a solution of \eqref{main} with $\mu=f$.
Then,
\begin{itemize}
\item[(a)] $u$ is bounded and
\[\|u\|_{L^\infty(\Omega)}\leq C\|f\|_{L^p(\Omega)}.\]
\item[(b)] If in addition $\Omega$ satisfies the uniform exterior ball  condition, then $u$ is $C(\overline\Omega)$. 
\end{itemize}
\end{lemma}

\begin{proof}
Without loss of generality, we assume that $f\geq0$ (if not, using linearity we can consider $|f|$ and then apply the result).

Let $\tilde u$ be the nonnegative weak solution in all of $\rn$ of 
\[\ls \tilde u=f\chi_\Omega\quad \text{in}\ \rn.\]

By \rife{dk}, we have
\[\|\tilde u\|_{C^\gamma(\rn)}\leq C\|f\|_{L^p(\Omega)},\]
where $\gamma=2s-\frac{n}{p}$.

Let us now consider $v=\tilde u-u$.
This function satisfies
\begin{equation}\label{caffa}\left\{ \begin{array}{rcll}
\ls v &=&0&\textrm{in }\Omega \\
v&=&\tilde u&\textrm{in }\rn\backslash\Omega,
\end{array}\right.\end{equation}
with $\tilde u\geq0$ and $\tilde u\in C^\gamma(\rn)$.
Let us prove now (a) and (b)

(a) By the maximum principle, since $\tilde u\geq0$ then we will have $v\geq0$.
Thus, 
\[0\leq u=\tilde u-v\leq \tilde u\leq  C\|f\|_{L^p(\Omega)},\]
and hence 
\[\|u\|_{L^\infty(\Omega)}\leq  C\|f\|_{L^p(\Omega)}.\]

(b) As  $\Omega$ satisfies the uniform exterior ball condition and $\tilde u$ is $C^\gamma(\rn)\cap L^{\infty}(\rn)$, then ${v}$,  the  solution of \rife{caffa}, is $C(\overline{\Omega})$ (see \cite{cs3}, \cite{kl}, and Remark \ref{valdi} above).

This concludes the proof of Lemma \ref{cont} as $u=\tilde u - v$.
\end{proof}

\begin{proof}[Proof of Theorem \ref{tei}]
In order to get the optimal regularity of $u$ we will consider less regular functions $g$ in Definition \ref{dualei}. The equivalence between the two definition relies on a straightforward density argument.  Let us fix $p>\frac{n}{2s}$.
For any $g\in L^{p}(\Omega)$, let us
define the following operator $T: L^{p}(\Omega)\mapsto \re$ through
$$
T(g):=\int_{\Omega} w(x)\ d\mu.
$$
Using Lemma \ref{cont}, $T$ is well defined, and we can write
$$
|T(g)| \leq \|w\|_{L^{\infty}(\Omega)}\ |\mu|(\Omega)\leq C\|g\|_{L^{p}(\Omega)},
$$
where $C$ depends only on $\Omega, \mu, n, s$ and $p$.
Then $T$ is a  bounded continuous
linear functional on  $L^{p}(\Omega)$, so that by
Riesz Representation Theorem, there
exists a unique function  $u\in L^{p'}(\Omega) $,  such that
\begin{equation}\label{sig}
    \int_{\Omega} w\ d\mu=\int_{\Omega}ug.
\end{equation}
{Of course, we can repeat the argument for any
$p>\frac{n}{2s}$  and we find a unique $u\in L^{p'}(\Omega)$   (and so $p' < \frac{n}{n-2s}$) such that (\ref{sig})  holds.}

Uniqueness easily follows by the fundamental theorem of calculus of variations as if $u$ and $v$ are two solutions in the sense of Definition \ref{dualei} then one has
$$
\into (u-v)g = 0
$$
for any $g\in C^{\infty}_{0}(\Omega)$  and so $u=v$.

\end{proof}

\begin{remark}\label{meta}
Let us notice the fact that our existence and uniqueness result for these general integro-differential operators does not rely at all on the knowledge of the fundamental solution.  Moreover, continuity of solution to problem \rife{main} with sufficiently smooth data is what is needed here in order to apply the duality method.  However, more precise regularity results up to the boundary can be found in the literature in some particular cases: e.g. the  fractional laplace case (\cite{rs2}) or the case of anisotropic $\alpha$-stable processes  (\cite{rs3},  see also Section \ref{alpha} below for a precise definition of these latter type of operators).

We also would like to stress a general fact about the method we used. As it is well known (and,  as one can easily deduce from the proof of Theorem \ref{tei}) the Stampacchia's duality method relies essentially on two main ingredients: linearity (in particular on the possibility to define an adjoint operator, $\mathcal{L}^{s}$ itself in Theorem \ref{tei} as $\mathcal{L}^{s}$ is self-adjoint),  and a regularity result. For the sake of exposition we restricted ourself on the  integro-differential case with symmetric kernels. Anyhow, it is clear that   a formally identical  existence and uniqueness metatheorem can be proved in the same way for very general boundary value problems involving (local or) nonlocal linear operators.  For a review on more general nonlocal operators (e.g. $x$-dependent kernels,  non-symmetric case, etc.) we refer for instance to \cite{fkv} and references therein. Roughly speaking, such a metatheorem reads as: { let $\mathcal{L}$ be a nonlocal linear operator, $\mu$ a bounded radon measure on $\Omega$, and let $\mathcal{G}^{*}$ be the Green operator for the homogeneous Dirichlet boundary value problem associated to the  adjoint operator $\mathcal{L}^{*}$. If $\mathcal{G}^{*}$ maps continuously $L^{q}(\Omega)$ into $C(\overline{\Omega})$, for any $q>q_{0}$, then there exist a unique duality solution for $u\in L^{p}(\Omega)$ for any $p<\frac{q_{0}}{q_{0}-1}$ for 
$$
    \begin{cases}
       {\mathcal  L} u = \mu &\quad \text{in $\Omega$},  \\
        u=0 \quad &\text{on} \ \ \rn\backslash\Omega\,.
    \end{cases}
$$ 
}
\end{remark}

\section{Further Remarks and Extensions}\label{frac}

This final section is devoted to present some further remarks on the regularity of the duality solutions found in Theorem \ref{tei}. One of the main tools in order to study Sobolev fractional regularity for solutions to Dirichlet integro-differential problems relies on the use of Bessel potentials associated with this kind of operators. A general treatment of this issue is out of the purpose of this note so, for the sake of concreteness, we will describe the method in some particular cases.

\subsection{Local regularity for $\alpha$-stable processes}
\label{alpha}
As one can imagine Sobolev fractional regularity of duality solutions can be deduced  if we have some informations on the exact behavior of the fundamental solution for this operator. For the sake of exposition  we consider the case of anisotropic $\alpha$-stable process and we readapt the result in \cite{kpu} in order to get sharp local Sobolev regularity  for the solution of \rife{main}.

Let $\Omega$ be a bounded open set of $\rn$. Consider
\begin{equation}\label{mainia}
\left\{ \begin{array}{rcll}
L^{\alpha} u &=&\mu&\textrm{in }\Omega \\
u&=&0&\textrm{in } \rn \backslash\Omega,
\end{array}\right.
\end{equation}
where $\mu$ is a radon bounded measure on  $\Omega$.
The operator $L^{\alpha}$ is given by
$$
L^{\alpha}u(x)= {\rm PV}\int_{\rn}\bigl(u(x)-u(x+y)\bigr)\frac{a\left(y/|y|\right)}{|y|^{n+\alpha}}dy,
$$
where $0<\alpha<2$ and $a:S^{n-1}\longrightarrow \re$ is a nonnegative and symmetric function that satisfies the uniform ellipticity condition
\[\lambda\leq a(\theta)\leq \Lambda\qquad \textrm{for all}\quad \theta\in S^{n-1},\]
with $0<\lambda\leq \Lambda$.

These operators are infinitesimal generators of a very special class of L\'evy processes: the so-called $\alpha$-stable processes. With respect to the previous section here we use the usual convention $\alpha={2s}$

It is proved in \cite{S} that the potential kernel $K$ associated to $L^{\alpha}$ (i.e., the fundamental solution of the operator) satisfies
\[\frac{c_{1}}{|y|^{n-\alpha}}\leq K(y)\leq \frac{c_{2}}{|y|^{n-\alpha}}\,,\]
for suitable positive constants $c_{1}\leq c_{2}$, and
\[K(y)=|y|^{\alpha-n}K\left(\frac{y}{|y|}\right).\]
Moreover, by Theorem 1 in \cite{BSS}, we have that $K$ belongs to the H\"older space $C^{2\alpha-\epsilon}_{loc}(\rn\setminus\{0\})$ for all $\epsilon>0$.

Thus, it follows that $K$ is a fractional kernel of order $\alpha$ and regularity $2\alpha-\epsilon$ in the sense of Definition 4.1 in \cite{GCG}. In particular, we have the following

\begin{theorem}\label{gcg}
The Riesz potential associated to  $L^{\alpha}$
\begin{equation}I^{\alpha}(f)(x)=\int_{\rn}f(y)K(x-y)dy \label{quela}\end{equation}
maps $L^p(\rn)$ boundedly into $C^{\alpha-\frac np}(\rn)$ whenever $p>n/\alpha$. That is, 
\[\| I^{\alpha}(f)\|_{C^{\alpha-\frac np}(\rn)}\leq C\|f\|_{L^p(\rn)},\]
with $C$ depending only on $p$, $n$, $\lambda$, $\Lambda$, and $\alpha$.
\end{theorem}
\begin{proof}
 see Theorem 5.2 in \cite{GCG}.
\end{proof}

Due to Theorem \ref{gcg}, and in particular to the integral representation \rife{quela} it is straightforward to reproduce the argument in \cite{kpu} in order to show the existence of a unique duality solution $\tilde{u}$  for problem
$$
L^{\alpha} u =\mu\ \textrm{in }\rn,
$$

whenever $\mu$ is set to be zero outside $\Omega$. Moreover, $\tilde{u}\in W^{1-\frac{2-\alpha}{q},q}_{loc}(\rn)$
for any $q<\frac{n+2-\alpha}{n+1-\alpha}$ 

For simplicity, in what follows we also assume $\alpha>1$, (i.e. $s>\frac{1}{2}$, see \cite{kpu}, Pag. 2,  and Remark \ref{alp} below for further comments on this restriction) and $\mu$ to be nonnegative. 

We have the following local result

\begin{theorem} The duality solution $u$ of problem \rife{mainia} belongs to $W^{1-\frac{2-\alpha}{q},q}_{loc}(\Omega)$
for any $q<\frac{n+2-\alpha}{n+1-\alpha}$. 
\end{theorem}

\begin{proof}

Let $\tilde u$ be the duality solution of
\[L^{\alpha} \tilde u=\mu\quad in\quad \rn,\]
where, as we said, $\mu$ is extended by zero outside $\Omega$.
Let us consider now $v:=\tilde u-u$.
This function is a nonnegative function  in all of $\rn$, and it satisfies  $(-\Delta)^sv=0$ in $\Omega$.
Moreover, $v=\tilde u$ outside $\Omega$, and in particular $\tilde u\in L^{\infty}(\rn\backslash\Omega)$.
that implies  that $v$ is locally smooth inside $\Omega$ (see for instance \cite{sil,cs2}).
So that, $v$ also satisfies that $v\in W^{1-\frac{2-\alpha}{q},q}_{loc}(\Omega)$
for any $q<\frac{n+2-\alpha}{n+1-\alpha}$, and hence the proof is complete as $u=\tilde u-v$. 
\end{proof}

\subsection{Towards a  global regularity result}

Let us come back to general integro-differential Dirichlet boundary value problem
\begin{equation}\label{maing}
\left\{ \begin{array}{rcll}
\mathcal{L}^{s} u &=&\mu&\textrm{in }\Omega \\
u&=&0&\textrm{in }\rn\backslash\Omega,
\end{array}\right.
\end{equation}
where $\mathcal{L}^{s}$ is of the type we considered in Section \ref{2}. 
  As we already said, a fine study of the Bessel potentials associated with these operators is out of our scopes here, so, for the sake of presentation, we are going to  straighten  the assumptions on both the operators $\mathcal{L}^{s}$ and on the admissible domains $\Omega$. We assume that $\mathcal{L}^{s}$ satisfies the following Calder\'on-Zygmund type property:  for $\nu\geq 0$ and $r> N+2-2s$,  there exist $C>0$ such that,
\begin{equation}\label{cz}
\| \mathcal{L}^{-s}g\|_{W^{2s-\nu,r}(\Omega)}\leq C\|g\|_{W^{-\nu,r}(\Omega)}
\end{equation}
for any $g\in W^{-\nu,r}(\Omega)$, where  $\mathcal{L}^{-s}g$  is the solution of
$$
\left\{ \begin{array}{rcll}
\mathcal{L}^{s} w &=&g&\textrm{in }\Omega \\
w&=&0&\textrm{in }\rn\backslash\Omega\,.
\end{array}\right.
$$

 Properties as \rife{cz}, which are established for instance in the fractional laplace case, are natural for general operators of fractional order $s$ at least for sufficiently smooth domain (see for instance \cite{ziem, dk}).  
 
Concerning the regularity of the domain,  for simplicity we may think at  $\partial\Omega\in \mathcal{C}^{1,1}$. Anyway, this assumption is not sharp. The same argument will work for more general bounded domains  of $\rn$ satisfying the extension property, the uniform ball condition and with no external cusps. Moreover, as the proof will  essentially be based on Sobolev embeddings, if $\mu\in L^{1}(\Omega)$ then one can remove the exterior ball assumption on $\Omega$ (e.g. Lipschitz domains). 
\begin{theorem}\label{furge}
Let $\mathcal{L}^{s}$ and $\Omega$ as above, then  the duality solution of problem \rife{maing} belongs to $u\in W^{\eta,q}_{0}(\Omega)$, for any $q<\frac{n+2-2s}{n+1-2s}$, where $\eta=1-\frac{2-2s}{q}$. 
\end{theorem}

\begin{remark}
\label{alp}
We notice that the derivation exponent
$\eta=1-\frac{2-2s}{q}$
can be negative. This is not the case if, for instance, $s>\frac{1}{2}$, so, for the sake of exposition we will understand this assumption in the following proof. However, the result is formally correct for any $s\in (0,1)$ once one interpret the space $W^{\eta,q} (\Omega)$ in the distributional sense as the dual of $W^{-\eta,q'}_{0}(\Omega)$ (see again \cite{dpv}). 

Finally, we want to emphasize that, as $s\to 1^-$, we \emph{formally} recover the classical optimal
summability of the gradient for linear elliptic boundary value problems with
measure data (see again \cite{s}), $u\in W^{1,q}_{0}(\Omega)$, for any $q<\frac{n}{n-1}$.
Also notice that, as before, the result is optimal, as it coincides with  the regularity 
of the fundamental
solution for the fractional Laplacian $G(x)=c|x|^{2s-n}$. 
\end{remark}

\begin{proof}
Let us fix $q$ as in the statement of Theorem \ref{furge} and let $\eta=1-\frac{2-2s}{q}$.  
Observe that $(2s-\eta)q'>n$ if and only if $q<\frac{n+2-2s}{n+1-2s}$. In particular, under this assumption we can use Theorem \ref{emb} in order to have
\[
 \|w\|_{L^{\infty}(\Omega)}\leq C\|w\|_{W^{2s-\eta,q'}(\Omega)}\,,     \qquad \forall w\in W^{2s-\eta, q'}(\Omega), 
 \]

Hence,  we reason as in the proof of Theorem \ref{tei}, and, using \rife{cz}, we have that, for any $g\in C^{\infty}_{0}(\Omega)$ 
$$
|T(g)| \leq \|w\|_{L^{\infty}(\Omega)}\ |\mu|(\Omega)\leq C\|w\|_{W^{2s-\eta,q'}(\Omega)}\leq C\|g\|_{W^{-\eta, q'}(\Omega)},
$$
from which we deduce that  $T$, defined as
$$
T(g) :=\int_{\Omega} w(x)\ d\mu=\int_{\Omega} u g\,,
$$
is a linear and continuous operator on $W^{-\eta, q'}(\Omega)$, that implies $u\in W^{\eta, q}_{0}(\Omega) $ as $u\equiv 0$ outside of $\Omega$.  
\end{proof}

\subsection*{Acknowledgement}
The author is deeply grateful to Xavier Ros-Oton for essential advices and fruitful discussions during the preparation of this manuscript. 
The author  is   partially supported by  the Gruppo Nazionale per l'Analisi Matematica, la Probabilit\`a e le loro Applicazioni (GNAMPA) of the Istituto Nazionale di Alta Matematica (INdAM).

\end{document}